\documentclass[12pt]{article}
\usepackage{BOONDOX-cal}
\usepackage{amssymb,a4}
\usepackage{amsmath,amsfonts,amssymb,amsthm, mathrsfs}
\usepackage{xcolor}

\allowdisplaybreaks[4]
\newtheorem{thm}{Theorem}[section]
\newtheorem{lem}[thm]{Lemma}
\newtheorem{defi}[thm]{Definition}
\newtheorem{prop}[thm]{Proposition}
\newtheorem{rem}[thm]{Remark}
\newtheorem{cor}[thm]{Corollary}

\newcommand{\fg}{\mathfrak{g}}
\newcommand{\fh}{\mathfrak{h}}
\newcommand{\fht}{\widehat{\mathfrak{h}}}
\newcommand{\fgt}{\widehat{\mathfrak{g}}}

\newcommand{\Ext}{\mathrm{Ext}}
\newcommand{\End}{\mathrm{End}}
\newcommand{\Ker}{\mbox{Ker}}
\newcommand{\C}{\mathbb{C}}
\newcommand{\Z}{\mathbb{Z}}
\newcommand{\Oinf}{\mathcal{O}^{\mathrm{inf}}}
\newcommand{\rh}{\widehat{\rho}}
\newcommand{\Hom}{\mathrm{Hom}}
\newcommand{\Ic}{{I_V}}
\newcommand{\KL}{{KL}}
\newcommand{\im}{\mbox{Im}}
\begin{document}

\begin{center}
{\Large\bf Ordinary modules for affine vertex operator superalgebras}
\end{center}

\begin{center}
{Huaimin Li, Qing Wang\\
School of Mathematical Sciences, Xiamen University, Xiamen, China 361005}	
\end{center}

\begin{abstract}
Let $\mathfrak{g}$ be a basic classical Lie superalgebra and
let $\widehat{\mathfrak{g}}$ be the corresponding affine Lie superalgebra.
In this paper, we first prove that a Cartan subalgebra acts semisimply on ordinary modules for the simple affine vertex operator superalgebra $L_{\widehat{\mathfrak{g}}}(k,0)$ at boundary admissible level $k$.
Then we prove that the category $\mathcal{O}_{k}^{ord}(\fg)$ of ordinary $L_{\widehat{\mathfrak{g}}}(k,0)$-modules is finite, semisimple and
$\mathcal{O}_{k}^{ord}(\fg)$ is exactly the category $\KL_k(\fg)$ of finite-length generalized modules for the affine vertex operator superalgebra $L_{\widehat{\mathfrak{g}}}(k,0)$.
Thus $\mathcal{O}_{k}^{ord}(\fg)$ is a braided tensor supercategory.
Furthermore, we obtain the rigidity of the supercategory $\mathcal{O}_{k}^{ord}(\fg)$
and thus it is a ribbon supercategory.
Finally, we conclude that $\mathcal{O}_{k}^{ord}(\fg)$ is a ribbon fusion supercategory.
\end{abstract}

\section{Introduction}

Huang and Lepowsky developed the tensor product theory of the module category for vertex operator algebra in \cite{HL95-1,HL95-2,HL95-3,H95}.
Furthermore, Huang proved that the category of modules for vertex operator algebra ${V}$ is a modular tensor category if
${V}$ is of CFT type, simple, self-dual, rational and $C_2$-cofinite in \cite{H08}.
Thus the category of modules for simple affine vertex operator algebra $L_{\widehat{\mathfrak{g}}}(k,0)$ at positive integer level $k$
is a modular tensor category.
If the level $k$ is not a positive integer, $L_{\widehat{\mathfrak{g}}}(k,0)$ is not $C_2$-cofinite \cite{DM06,Z96}.
Level $k$ is called an admissible level if simple affine vertex operator superalgebra $L_{\widehat{\mathfrak{g}}}(k,0)$ is an admissible $\widehat{\mathfrak{g}}$-module,
where admissible modules were first studied by Kac and Wakimoto to explore the modular invariant representations of affine Lie (super)algebras in \cite{KW88}.
For $\mathfrak{g}=sl_2$ and admissible level $k$,
Adamovi\'{c} and Milas proved that
$L_{\widehat{\mathfrak{g}}}(k,0)$ is rational in the category $\mathcal{O}$ and presented a conjecture for arbitrary affine Lie algebras in \cite{AM95}.
This conjecture was proved by Arakawa in \cite{Ar16} and
 extended to the only non-isotropic basic Lie superalgebra $\mathfrak{g}=osp(1|2n)$ by Gorelik and Serganova in \cite{GS22}.
On the other hand,
Creutzig, Huang and Yang proved that the ordinary module category of $L_{\widehat{\mathfrak{g}}}(k,0)$
is a braided tensor category for any simple finite-dimensional Lie algebra $\mathfrak{g}$ at admissible level $k$ in \cite{CHY18}.
In \cite{C19}, Creutzig proved that the ordinary module category of $L_{\widehat{\mathfrak{g}}}(k,0)$ associated to any simply laced Lie algebra $\mathfrak{g}$ at admissible level $k$ is rigid and thus a braided fusion category.
By using the theory of vertex superalgebra extensions \cite{CKM},
Creutzig, Genra and Linshaw proved that the ordinary module category of $L_{\widehat{osp(1|2n)}}(k,0)$ at admissible level $k$ is a fusion supercategory in \cite{CGL24}.

While besides the simple finite-dimensional Lie algebras and $osp(1|2n)$,
modular invariance also occurs in boundary level admissible modules when $\fg$ is the other basic classical Lie superalgebra \cite{GK15,KW17}.
As a representative,
we studied the simple affine vertex operator superalgebra $L_{\widehat{sl(2|1)}}(k,0)$ at boundary admissible level in \cite{LW25}.
More specifically, we proved that $L_{\widehat{sl(2|1)}}(-\frac{1}{2},0)$ is rational in the category $\mathcal{O}$ at the boundary admissible level $-\frac{1}{2}$,
while at the non-boundary admissible level $\frac{1}{2}$ we classified its infinitely many irreducible weak modules in the category $\mathcal{O}$,
and we gave a conjecture that for any basic classical Lie superalgebra $\mathfrak{g}$,
$L_{\widehat{\mathfrak{g}}}(k,0)$ is rational in the category $\mathcal{O}$ at boundary admissible level $k$ and the irreducible weak modules in the category $\mathcal{O}$ are exactly the admissible modules of level $k$ for $\widehat{\mathfrak{g}}$.
In \cite{GK25}, Gorelik and Kac proved that $L_{\widehat{\mathfrak{g}}}(k,0)$ is rational in the category $\mathcal{O}$
for any basic classical Lie superalgebra except $D(2|1;a)$ at boundary admissible level $k$.
Based on their results, we classified the irreducible ordinary modules of $L_{\widehat{\mathfrak{g}}}(k,0)$
at boundary admissible level $k=\frac{h^{\vee}}{u}-h^{\vee}$,
we proved that if $\mathfrak{g}$ is a basic classical Lie superalgebra of type I,
$L_{\widehat{\mathfrak{g}}}(k,0)$ has exactly $u$ pairwise inequivalent irreducible ordinary modules $L(\widehat\lambda_p)$, $p=0,1,\dots,u-1$,
if $\mathfrak{g}$ is a finite-dimensional simple Lie algebra or a basic classical Lie superalgebra of type II except $D(2|1;a)$,
$L_{\widehat{\mathfrak{g}}}(k,0)$ itself is the only irreducible ordinary $L_{\widehat{\mathfrak{g}}}(k,0)$-module \cite{LW26}.

Let $\mathfrak{g}=\mathfrak{g}_{\bar{0}}\oplus\mathfrak{g}_{\bar{1}}$ be a basic classical Lie superalgebra and let $\fh$ be a Cartan subalgebra of $\fg$,
we fix simple roots of $\fg$ to be the distinguished simple roots.
In this paper, we study the category of ordinary $L_{\widehat{\mathfrak{g}}}(k,0)$-modules and denote it by $\mathcal{O}_{k}^{ord}(\fg)$,
where $L_{\widehat{\mathfrak{g}}}(k,0)$ is the simple affine vertex operator superalgebra at boundary admissible level $k$.
We first prove that $\fh(0)$ acts semisimply on ordinary $L_{\widehat{\mathfrak{g}}}(k,0)$-modules.
If $\mathfrak{g}$ is a basic classical Lie superalgebra of type II
or a finite-dimensional simple Lie algebra, since the even part $\fg_{\bar{0}}$ is semisimple,
we obtain that $\fh(0)$ acts semisimply on ordinary $L_{\widehat{\mathfrak{g}}}(k,0)$-modules immediately by Weyl's completely reducibility theorem.
While for the type I case, note that $\widehat{psl(n|n)}$ does not admit boundary admissible levels,
we assume $\mathfrak{g}\ne psl(n|n)$ in this paper,
then $\fg_{\bar{0}}=\fg_{\bar{0}}^{\mathrm{ss}}\oplus\C z$, where $\fg_{\bar{0}}^{\mathrm{ss}}=[\fg_{\bar{0}},\fg_{\bar{0}}]$ is semisimple and $\C z=Z(\fg_{\bar{0}})$ is the center,
the main difficulty is to prove that $z(0)$ acts semisimply on ordinary $L_{\widehat{\mathfrak{g}}}(k,0)$-modules.
Let $z(0)=z_s(0)+z_n(0)$ be the Jordan-Chevalley decomposition on an ordinary $L_{\widehat{\mathfrak{g}}}(k,0)$-module,
we first prove that the nilpotent part $z_n(0)$ is a $L_{\widehat{\mathfrak{g}}}(k,0)$-module endomorphism.
Suppose $z_n(0)$ is non-zero,
we show that there exists a self-extension
$0\longrightarrow L(\widehat\lambda_p)\longrightarrow E\longrightarrow L(\widehat\lambda_p)\longrightarrow0$
in $\mathcal{O}_{k}^{ord}(\fg)$ such that $z_n(0)$ on $E$ is non-zero.
From Subsection 1.6 of \cite{GK11},
there exists an injective map
$\Upsilon_{\widehat{\lambda}_p}:\Ext^1_{\fgt}\big(L(\widehat{\lambda}_p),L(\widehat{\lambda}_p)\big)\longrightarrow\fht^*$,
and set $\Upsilon_{\widehat{\lambda}_p}(E)=\nu$.
Let $\gamma_p=\alpha_{odd}+p\delta$,
where $\alpha_{odd}$ is the odd isotropic simple root of $\mathfrak{g}$ and $\delta$ is the imaginary root of $\widehat{\mathfrak{g}}$,
we can prove $\nu$ is orthogonal to $\gamma_p$.
On the other hand we have $(\nu,\gamma_p)\neq0$ by a direct calculation,
thus a contradiction.
Hence $z(0)=z_s(0)$ acts semisimply on ordinary $L_{\widehat{\mathfrak{g}}}(k,0)$-modules,
and then $\fh(0)$ acts semisimply on ordinary $L_{\widehat{\mathfrak{g}}}(k,0)$-modules.
Together with the classification of simple ordinary modules \cite{LW26} and the complete reducibility of weak $L_{\widehat{\mathfrak{g}}}(k,0)$-modules in the category $\mathcal{O}$ \cite{GK25},
we show that the category $\mathcal{O}_{k}^{ord}(\fg)$ is finite and semisimple.
From Corollary 3.13 in \cite{CY}, we know that
if the category of ordinary $L_{\fgt}(k,0)$-modules is semisimple,
then the category $\KL_k(\fg)$ of finite-length generalized modules has a natural structure of braided tensor category,
and this result naturally generalizes to affine vertex operator superalgebras \cite{CMY22}.
To prove $\mathcal{O}_{k}^{ord}(\fg)$ is a braided tensor supercategory,
we prove $\KL_k(\fg)=\mathcal{O}_{k}^{ord}(\fg)$,
the strategy is to prove $\Ext^1_{\KL_k(\fg)}(L,L)=0$ for any irreducible ordinary $L_{\widehat{\mathfrak{g}}}(k,0)$-module $L$.
Since $\mathcal{O}_{k}^{ord}(\fg)$ is a braided tensor supercategory
and all structure morphisms of $\mathcal{O}_{k}^{ord}(\fg)$ are even \cite{CKM},
the even part $\mathcal{O}_{k}^{ord}(\fg)^{\mathrm{ev}}$
is a braided tensor category, where $\mathcal{O}_{k}^{ord}(\fg)^{\mathrm{ev}}$ is the subcategory of $\mathcal{O}_{k}^{ord}(\fg)$
with the same objects and with parity-preserving homomorphisms as morphisms.
Furthermore, we show that $\mathcal{O}_{k}^{ord}(\fg)^{\mathrm{ev}}$ is an $r$-category in the sense of \cite{BD}.
From \cite{EP}, every semisimple braided $r$-category of moderate growth is rigid,
we obtain the rigidity of the category $\mathcal{O}_{k}^{ord}(\fg)^{\mathrm{ev}}$ and thus
$\mathcal{O}_{k}^{ord}(\fg)$ is a finite semisimple rigid braided tensor supercategory with the ribbon twist $\theta=e^{2\pi iL(0)}$, i.e., a ribbon fusion supercategory.

The paper is organized as follows.
In Section~2, we first review some results about affine Lie superalgebras from \cite{GK11,GK25} and prove a technical lemma.
Then we recall the classification result of irreducible ordinary modules
for simple affine vertex operator superalgebra $L_{\widehat{\mathfrak{g}}}(k,0)$ at boundary admissible level $k$,
by this classification result we prove that every ordinary module of $L_{\widehat{\mathfrak{g}}}(k,0)$ at boundary admissible level $k$ has finite length.
In Section~3, we prove the semisimplicity of $\fh(0)$ on ordinary $L_{\widehat{\mathfrak{g}}}(k,0)$-modules,
and obtain that $\mathcal{O}_{k}^{ord}(\fg)$ is finite and semisimple.
In Section~4, we prove that $\mathcal{O}_{k}^{ord}(\fg)$ is a rigid braided tensor supercategory with the ribbon twist $\theta=e^{2\pi iL(0)}$, and thus $\mathcal{O}_{k}^{ord}(\fg)$ is a ribbon fusion supercategory.

\section{Preliminaries}

In this section, we first review some results about affine Lie superalgebras from \cite{GK11,GK25} and prove the technical Lemma \ref{lem:tangent}.
Then we recall the classification result of irreducible ordinary modules
for simple affine vertex operator superalgebra $L_{\widehat{\mathfrak{g}}}(k,0)$ at boundary admissible level $k$,
and prove that every ordinary module of $L_{\widehat{\mathfrak{g}}}(k,0)$ at boundary admissible level $k$ has finite length.

\subsection{Affine Lie superalgebra}

Let $\fg$ be a basic classical Lie superalgebra and let
$\fgt=\fg\otimes\C[t,t^{-1}]\oplus\C K\oplus\C d$ be the corresponding non-twisted affine Lie superalgebra.
Let $\fh$ be a Cartan subalgebra of $\fg$ and let $(\cdot,\cdot)$ be
the nondegenerate even supersymmetric invariant bilinear form
which is normalized by $h^\vee\geq0$ and $(\alpha,\alpha)=2$ for a root of maximal square length,
where $h^{\vee}$ is the dual Coxeter number, i.e., half of the eigenvalue of the Casimir operator.
We fix simple roots of $\fg$ as the distinguished simple roots, and let $\Pi=\{\alpha_1,\cdots,\alpha_n\}$ be the set of simple roots, let $\theta$ be the highest root,
let $\rho$ be the Weyl vector of $\mathfrak{g}$ with respect to $\Pi$, i.e., $(\rho,\alpha_i)=\frac{1}{2}(\alpha_i,\alpha_i)$ for any $i=1,\cdots,n$.
Let $\delta$ be the linear function on $\widehat{\mathfrak{h}}$ defined by $\delta|_{\mathfrak{h}\oplus\mathbb{C}K}=0, \delta(d)=1$.
Then $\widehat{\Pi}=\{\alpha_0=\delta-\theta,\alpha_1,\cdots,\alpha_n\}$ is a set of simple roots of $\widehat{\mathfrak{g}}$.
Note that $\widehat{\rho}=h^{\vee}\Lambda_0+\rho$ is the Weyl vector of $\widehat{\mathfrak{g}}$ with respect to $\widehat{\Pi}$.
Define the fundamental weights $\Lambda_i\in \widehat{\mathfrak{h}}^* (i=0,1,\cdots,n)$ by $\Lambda_i(\alpha_j^{\vee})=\delta_{ij}, \Lambda_i(d)=0$ for any $i,j\in\{0,1,\cdots,n\}$.
For $\beta\in \mathfrak{h}^*$, define the translation $t_\beta\in \mbox{End}~\widehat{\mathfrak{h}}^*$ by
$t_\beta({\lambda})={\lambda}+{\lambda}(K)\beta-(({\lambda},\beta)+\frac{1}{2}{\lambda}(K)(\beta,\beta))\delta.$
For any ${\lambda}\in \widehat{\mathfrak{h}}^*$,
denote by $M({\lambda})$ (resp. $L({\lambda})$) the Verma (resp. irreducible highest weight) $\widehat{\mathfrak{g}}$-module.
Denote by $s_\alpha$ the reflection with respect to an even root $\alpha$.

Let $\Delta$ (resp. $\Delta^{\mathrm{re}}$) be the root system (resp. real root system) of $\fgt$ and
set $\bar{\Delta}^{\mathrm{re}}=\{\alpha\in\Delta^{\mathrm{re}}\mid \frac{1}{2}\alpha\notin\Delta^{\mathrm{re}}\}$.
For any $\lambda\in\fht^*$,
let $\bar{R}_{\lambda+\widehat{\rho}}$ be the set of $\alpha\in\bar{\Delta}^{\mathrm{re}}$ with
$2(\lambda+\widehat{\rho},\alpha)\in\Z(\alpha,\alpha)$ if $2\alpha\notin \Delta$ and
$2(\lambda+\widehat{\rho},\alpha)\in(2\Z+1)(\alpha,\alpha)$ if $2\alpha\in \Delta$.
Set $R_{\lambda+\widehat{\rho}}=\bar{R}_{\lambda+\widehat{\rho}}\cup\{2\alpha\mid\alpha\in\bar{R}_{\lambda+\widehat{\rho}},2\alpha\in\Delta\}$ and
let $\Delta_{\lambda+\widehat{\rho}}\subseteq\Delta^{\mathrm{re}}$ be the minimal subset satisfying the following properties:
(i) $\bar{R}_{\lambda+\widehat{\rho}}\subset\Delta_{\lambda+\widehat{\rho}}$;
(ii) $s_\alpha\beta\in\Delta_{\lambda+\widehat{\rho}}$ for any $\alpha,\beta\in \Delta_{\lambda+\widehat{\rho}}$ such that $(\alpha,\alpha)\ne0$;
(iii) $\Delta_{\lambda+\widehat{\rho}}\cap\{\alpha\pm \beta\}=\Delta^{\mathrm{re}}\cap\{\alpha\pm \beta\}$ for any $\alpha,\beta\in \Delta_{\lambda+\widehat{\rho}}$ such that
$(\alpha,\alpha)=0\ne (\beta,\alpha)$;
(iv) if $\alpha\in\Delta_{\lambda+\widehat{\rho}}$ and $2\alpha\in\Delta^{\mathrm{re}}$, then $2\alpha\in\Delta_{\lambda+\widehat{\rho}}$.
Let $\Delta_{\lambda+\widehat{\rho}}^{+}$ be the positive root set of $\Delta_{\lambda+\widehat{\rho}}$.
An element $\lambda\in\fht^*$ is called {\em non-critical} if $(\lambda+\widehat{\rho},\delta)\ne 0$.
For any non-critical $\lambda$, let $U(\lambda)\subseteq\fht^*$
be the set of weights $\nu$ for which there exists a chain
$\nu+\widehat{\rho}=\mu_r<\mu_{r-1}<\cdots<\mu_0=\lambda+\widehat{\rho}$
with $\mu_{i+1}=\mu_i-m_i\gamma_i$ for some $\gamma_i\in\Delta^{\mathrm{re}}\cap\Delta^+$
and $m_i\in\Z_{>0}$ satisfying $2(\mu_i,\gamma_i)=m_i(\gamma_i,\gamma_i)$,
where $m_i$ is odd if $\gamma_i$ is odd non-isotropic and $m_i=1$ if
$\gamma_i$ is isotropic.

\begin{lem}[\cite{GK25}]\label{lem:U} Let $\lambda\in\fht^*$ be non-critical.\\
{\rm (i)} If $\mu\in\fht^*$ such that $[M(\lambda):L(\mu)]\neq0$, then $\mu\in U(\lambda)$,
$\Delta_{\mu+\widehat{\rho}}=\Delta_{\lambda+\widehat{\rho}}$ and
$\lambda-\mu\in\Z_{\ge0}\Delta_{\lambda+\widehat{\rho}}$. \\
{\rm (ii)} If $\alpha$ is a real isotropic root such that $(\lambda+\widehat{\rho},\alpha)=0$, then $\Delta_{\lambda+\widehat{\rho}}=\Delta_{\lambda+\widehat{\rho}+\alpha}$.
\end{lem}

\begin{defi}\label{defi:lmin}
{\rm Let $\lambda\in\fht^*$ be non-critical. Define
\[
C'(\lambda)\subset\Big\{(\alpha,m)\big|\ \alpha\in\Delta^{\mathrm{re}}_+,\ m\in\Z_{>0},\
2(\lambda+\widehat{\rho},\alpha)=m(\alpha,\alpha)\Big\}
\]
such that $m$ is odd if $\alpha$ is odd non-isotropic and $m=1$ if
$\alpha$ is isotropic, and $\frac{1}{2}\alpha$ is not an odd root if $\alpha$ is even. A real
root $\alpha$ with $(\alpha,m)\in C'(\lambda)$ is called \emph{$\lambda$-minimal} if
$\big[M(\lambda-n\beta):L(\lambda-m\alpha)\big]=0$
for all $(\beta,n)\in C'(\lambda),\ (\beta,n)\neq(\alpha,m)$.}
\end{defi}

Let $\lambda\in\fht^*$ be non-critical and let $L(\lambda)$ be the irreducible highest
weight $\fgt$-module with highest weight $\lambda$.
For exact sequence $0\longrightarrow L(\lambda)\stackrel{\phi_1}{\longrightarrow} E\stackrel{\phi_2}{\longrightarrow} L(\lambda)\longrightarrow0$,
let $v:=\phi_1(v_\lambda)$ and fix $v^\prime\in E_\lambda$ such that $\phi_2(v^\prime)=v_\lambda$,
where $v_\lambda$ is the highest weight vector of $L(\lambda)$ and $E_\lambda$ is the generalized weight space of $E$ with weight $\lambda$.
It is clear that $v,v^\prime$ is a basis of $E_\lambda$.
From Subsection 1.6 of \cite{GK11},
there exists a unique $\nu\in\fht^*$ such that
$hv^\prime=\lambda(h)v^\prime+\nu(h)\,v$ for any $h\in\fht$, we then define
\[
\Upsilon_\lambda:\Ext^1_{\fgt}\big(L(\lambda),L(\lambda)\big)\longrightarrow\fht^*,
\qquad \Upsilon_\lambda(E)=\nu.
\]

\begin{thm}[\cite{GK11}]\label{thm:Upsilon}
Let $\lambda\in\fht^*$ be non-critical.
Then the map $\Upsilon_\lambda$ is injective and
\[
\im\Upsilon_\lambda\subseteq\Big\{\mu\in\fht^*\big|\ (\mu,\alpha)=0
\ \text{for every } \lambda\text{-minimal root } \alpha\Big\}.
\]
\end{thm}

We call $\gamma\in\Delta_{\lambda+\widehat{\rho}}^{+}$ indecomposable in $\Delta_{\lambda+\widehat{\rho}}^{+}$ if
$\gamma$ is not a non-trivial positive integral combination
of positive roots of $\Delta_{\lambda+\widehat{\rho}}^{+}$.
Then we have the following lemma.

\begin{lem}\label{lem:tangent}
Let $\lambda\in\fht^*$ be
non-critical, and let $\gamma\in\Delta_{\lambda+\widehat{\rho}}^{+}$ be
isotropic and indecomposable in $\Delta_{\lambda+\widehat{\rho}}^{+}$ such that $(\lambda+\widehat{\rho},\gamma)=0$. If
$0\longrightarrow L(\lambda)\longrightarrow E\longrightarrow L(\lambda)\longrightarrow0$
is an exact sequence with $\Upsilon_\lambda(E)=\nu\in\fht^*$, then
$(\nu,\gamma)=0$.
\end{lem}

\begin{proof}
We want to prove that $\gamma$ is
$\lambda$-minimal, then by Theorem~\ref{thm:Upsilon}, we have $(\nu,\gamma)=(\Upsilon_\lambda(E),\gamma)=0$.
Suppose $\gamma$ is not $\lambda$-minimal, then there exists $(\eta,r)\in C'(\lambda)$
with $(\eta,r)\neq(\gamma,1)$ such that
$\big[M(\lambda-r\eta):L(\lambda-\gamma)\big]\neq0$.
Since $(\eta,r)\in C'(\lambda)$, $\eta\in R_{\lambda+\widehat{\rho}}\subseteq
\Delta_{\lambda+\widehat{\rho}}$, we have $\eta\in\Delta_{\lambda+\widehat{\rho}}^{+}$.
If $\eta$ is non-isotropic, then
$2(\lambda+\widehat{\rho},\eta)=r(\eta,\eta)$ and $\lambda-r\eta+\widehat{\rho}=s_{\eta}(\lambda+\widehat{\rho})$,
hence $\eta\in\Delta_{\lambda+\widehat{\rho}}$ implies that $\Delta_{\lambda+\widehat{\rho}}$ is invariant under $s_{\eta}$, i.e.,
$\Delta_{\lambda-r\eta+\widehat{\rho}}=s_{\eta}\Delta_{\lambda+\widehat{\rho}}=\Delta_{\lambda+\widehat{\rho}}$. If $\eta$ is isotropic, then
$(\lambda+\widehat{\rho},\eta)=0$, by Lemma~\ref{lem:U} (ii), we have
$\Delta_{\lambda+\widehat{\rho}}=\Delta_{\lambda-r\eta+\widehat{\rho}}$.
By Lemma~\ref{lem:U} (i), $\big[M(\lambda-r\eta):L(\lambda-\gamma)\big]\neq0$ implies $\lambda-\gamma\in U(\lambda-r\eta)$, i.e.\
there exist $\eta_1,\dots,\eta_s\in\Delta_{\lambda+\widehat{\rho}}^{+}$ and $m_i\in\Z_{>0}$ such that
$\lambda-\gamma=\lambda-r\eta-\sum_{i=1}^{s}m_i\eta_i$,
where $\eta_i\in\Delta_{\lambda+\widehat{\rho}}^{+}$. Hence
$\gamma=r\eta+\sum_{i=1}^{s}m_i\eta_i$
is a positive integral combination of roots of $\Delta_{\lambda+\widehat{\rho}}^{+}$ and the sum of the
coefficient is $r+\sum_{i=1}^{s}m_i\ge 2$ because $(\eta,r)\neq(\gamma,1)$.
This contradicts with the assumption that $\gamma$ is indecomposable in $\Delta_{\lambda+\widehat{\rho}}^{+}$. Thus $\gamma$ is
$\lambda$-minimal.
\end{proof}

\begin{lem}[\cite{GK11}]\label{lem:split}
Let $0\to L(\lambda)\to E\to L(\lambda)\to0$ be an exact sequence
with $\Upsilon_\lambda(E)=\nu\in\fht^*$. Then $0\to L(\lambda)\to E\to L(\lambda)\to0$ splits over the derived algebra $[\fgt,\fgt]$ if and only if
$\nu\in\Delta^{\perp}=\{\mu\in\fht^*~|~(\mu,\alpha)=0~\mbox{for all}~\alpha\in\Delta\}$.
\end{lem}

\subsection{Ordinary modules for simple affine vertex operator superalgebra $L_{\fgt}(k,0)$}

\begin{defi}
{\em  Let $V$ be a $\mathbb{Z}$-graded vertex operator superalgebra.
A {\em weak $V$-module} is a module for $V$ as a vertex superalgebra.
A {\em generalized $V$-module} is a weak $V$-module $M$ with a $\mathbb{C}$-grading $M=\coprod_{\lambda\in\mathbb{C}}M_{\lambda}$
such that $M_{\lambda}$ for $\lambda\in\mathbb{C}$ are generalized eigenspaces for the operator $L(0)$ with eigenvalues $\lambda$.
A {\em grading-restricted generalized $V$-module} is a generalized $V$-module such that for any $\lambda\in\mathbb{C}$,
$M_{\lambda+m}=0$ for $m\in\mathbb{Z}$ sufficiently negative and $\mbox{dim} M_{\lambda}<\infty$.
An {\em ordinary $V$-module} is a grading-restricted generalized $V$-module such that $M_{\lambda}$ for $\lambda\in\mathbb{C}$
are eigenspaces for the operator $L(0)$. A generalized $V$-module $M$ is of length $l$ if there exist generalized $V$-submodules
$0=N_0\subset N_1\subset\cdots\subset N_l=M$ such that $N_{i+1}/N_i$ for $i=0,\cdots,l-1$ are irreducible ordinary $V$-modules.}
\end{defi}

\begin{rem}
{\rm As a module for a vertex operator superalgebra, every weak $V$-module $M$ carries a parity decomposition $M=M^{\bar{0}}\oplus M^{\bar{1}}$ compatible with the vertex operator structure. Accordingly, for any $V$-modules $M,N$, the space $\Hom(M,N)$ of $V$-module homomorphisms is a superspace, i.e., the direct sum of the spaces $\Hom^{0}(M,N)$ and $\Hom^{1}(M,N)$ of parity-preserving and parity-reversing homomorphisms.}
\end{rem}

All boundary admissible levels of indecomposable non-twisted affine Kac-Moody superalgebras $\widehat{\mathfrak{g}}$ except $D(2|1;a)^{(1)}$ are classified in \cite{GK25}.
The boundary principal admissible level of $\widehat{\mathfrak{g}}$ has the form
${k}=\frac{h^{\vee}}{u}-h^{\vee}$, where $u$ is a positive integer,
$r^\vee$ is lacity of $\widehat{\mathfrak{g}}^\sharp$ (the ratio of square lengths of a long root and of a short root of $\Delta^\sharp=\{\alpha\in\Delta^0~|~(\alpha,\alpha)>0\}$) such that $\mbox{gcd}(u,r^\vee)=\mbox{gcd}(u,h^\vee)=1$ (if $h^\vee$ is not integral, then $(u,h^\vee)=1$ stands for $(2u,2h^\vee)=1$).
The boundary subprincipal admissible level happens only for $\mathfrak{g}=osp(1|2n)$.
Let $L_{\widehat{\mathfrak{g}}}(k,0)$ be the simple affine vertex operator superalgebra associated to $\fg$ at level $k$.
We recall the classification result of irreducible ordinary $L_{\widehat{\mathfrak{g}}}(k,0)$-module at boundary principal admissible level $k$.

\begin{thm}[\cite{LW26}]\label{thm:LW}
Let $k=\tfrac{h^{\vee}}{u}-h^{\vee}$ be a boundary principal admissible level.\\
{\rm (i)} If $\fg$ is a basic classical Lie superalgebra of type I, then
$L_{\widehat{\mathfrak{g}}}(k,0)$ has exactly $u$ pairwise inequivalent
irreducible ordinary modules
$L(\widehat\lambda_p), p=0,1,\dots,u-1$,
where for $\fg=\mathfrak{sl}(n|m)$,
$\widehat\lambda_p=\Big(\frac{p+1}{u}-1\Big)h^{\vee}\Lambda_0-\frac{p\,h^{\vee}}{u}\Lambda_n$,
and for $\fg=\mathfrak{osp}(2|2n)$ one replaces $\Lambda_n$ by $\Lambda_1$.\\
{\rm (ii)} If $\fg$ is a finite-dimensional simple Lie algebra or a basic classical Lie
superalgebra of type II except $D(2|1;a)$, then $L(\widehat\lambda_0)$ is the only
irreducible ordinary $L_{\widehat{\mathfrak{g}}}(k,0)$-module, where $\widehat\lambda_0=k\Lambda_0$.
\end{thm}

Let $h_p$ be the lowest conformal weight of $L(\widehat\lambda_p)$ and set $H=\{h_0,\dots,h_{u-1}\}$.

\begin{prop}\label{prop:finlen}
Let $\fg$ be a basic classical Lie superalgebra and let $k$ be a boundary principal admissible level of $\widehat{\mathfrak{g}}$.
Then every ordinary $L_{\widehat{\mathfrak{g}}}(k,0)$-module is of finite length.
\end{prop}

\begin{proof}
Let $M$ be an ordinary $L_{\widehat{\mathfrak{g}}}(k,0)$-module.
Take $0\neq v\in M$ and consider the cyclic
submodule $M'=L_{\widehat{\mathfrak{g}}}(k,0)v$, then $M'$ is ordinary.
If $\{W_i\}$ is a chain of proper submodules of $M'$, then $\bigcup_i W_i$ is a submodule and it
is proper, otherwise $v\in W_j$ for some $j$ and $M'=L_{\widehat{\mathfrak{g}}}(k,0)v\subseteq W_j$,
a contradiction.
By Zorn's lemma, $M'$ has a maximal proper submodule $M''$, and $M'/M''$ is an
irreducible $L_{\widehat{\mathfrak{g}}}(k,0)$-module. As a quotient of the ordinary module $M'$, the module $M'/M''$
is ordinary, so by Theorem~\ref{thm:LW} it is isomorphic to some $L(\widehat\lambda_p)$.
Then $M$ admits an irreducible
subquotient isomorphic to some $L(\widehat\lambda_p)$.
Define
$\ell(M):=\sum_{h\in H}\dim M_h\in\Z_{\ge0}$.
Since $H$ is finite and each conformal weight space of $M$ is finite-dimensional,
we have $\ell(M)<\infty$. If $N\subsetneq N'\subseteq M$ are submodules, then $N'/N$ is a non-zero
ordinary $L_{\widehat{\mathfrak{g}}}(k,0)$-module, it has an irreducible subquotient
isomorphic to $L(\widehat\lambda_p)$, hence $(N'/N)_{h_p}\neq0$ for some $p$, so $\ell(N'/N)\ge1$.
Since $L(0)$ acts semisimply, $\ell$ is additive on short exact sequences, then
$\ell(N')=\ell(N)+\ell(N'/N)$. Therefore $\ell(N')>\ell(N)$, and any strictly
increasing chain $0=N_0\subsetneq N_1\subsetneq\cdots\subsetneq N_r=M$ satisfies
$\ell(M)=\sum_{i=1}^{r}\ell(N_i/N_{i-1})\ge r$. Thus $r\le\ell(M)<\infty$, i.e., $M$ is of
finite length.
\end{proof}

\begin{rem}
{\rm Let $\mathfrak{g}=osp(1|2n)$ and let $k$ be a boundary subprincipal admissible level of $\fgt$.
From Theorem 4.2 in \cite{CGL24}, the number of irreducible ordinary $L_{\widehat{\mathfrak{g}}}(k,0)$-modules is finite
(see also \cite{GK25}). Similar to Proposition \ref{prop:finlen}, every ordinary $L_{\widehat{\mathfrak{g}}}(k,0)$-module is of finite length.}
\end{rem}

For an ordinary $L_{\widehat{\mathfrak{g}}}(k,0)$-module $M$ and $\lambda\in\C$,
since $[L(0),x(0)]=0$ for all $x\in\fg$,
 the space $M_{\lambda}$ is a
finite-dimensional $\fg$-module. Let
$\fg_{\bar{0}}^{\mathrm{ss}}$ be the semisimple part of the even subalgebra $\fg_{\bar{0}}$, and set
$\fh^{\mathrm{ss}}=\fh\cap\fg_{\bar{0}}^{\mathrm{ss}}$.
By Weyl's complete reducibility theorem, each $M_{\lambda}$ is a
completely reducible module as a $\fg_0^{\mathrm{ss}}$-module,
hence its Cartan subalgebra $\fh^{\mathrm{ss}}(0)$ acts semisimply on $M_{\lambda}$.
Then we have the following lemma.

\begin{lem}\label{lem:central}
Let $\fg$ be a basic classical Lie superalgebra and let $M$ be an
ordinary $L_{\widehat{\mathfrak{g}}}(k,0)$-module, then $\fh^{\mathrm{ss}}(0)$ acts semisimply
on $M$.
\end{lem}

Let $\fg$ be a basic classical Lie superalgebra of type I, there is a $\Z$-grading
$\fg=\fg_{-1}\oplus\fg_0\oplus\fg_1$, where $\fg_0=\fg_{\bar{0}}$.
Note that if $\mathfrak{g}$ is of type $A(n,n)$, then $h^\vee=0$,
$\widehat{\mathfrak{g}}$ does not admit boundary admissible levels,
we assume $\mathfrak{g}\ne psl(n|n)$ in this paper.
Then $\fg_0=\fg_{\bar{0}}^{\mathrm{ss}}\oplus\C z$ and $\fh=\fh^{\mathrm{ss}}\oplus\C z$,
where $z$ spans the center of $\fg_0$ and is normalized so that $[z,x]=jx$ for $x\in\fg_j$. By
Lemma~\ref{lem:central}, the semisimplicity of $\fh(0)$ is equivalent to that of
$z(0)$.

\begin{lem}\label{lem:Nend}
Let $\fg$ be a basic classical Lie superalgebra of type I. Let $M$ be an ordinary $L_{\fgt}(k,0)$-module and let $z(0)=z_s(0)+z_n(0)$ be the Jordan-Chevalley decomposition of
$z(0)$ on $M$. Then
$z_n(0)\in\End_{L_{\widehat{\mathfrak{g}}}(k,0)}(M)$.
\end{lem}

\begin{proof}
Since $[L(0),z(0)]=0$, $z(0)$ preserves each finite-dimensional conformal
weight space $M_{\lambda}$, so the Jordan-Chevalley decomposition exists on each $M_{\lambda}$ and
induces a global decomposition $z(0)=z_s(0)+z_n(0)$ on $M$. Fix $x\in\fg_j$, $n\in\Z$ and
$v\in M_{\lambda}^{(c)}$, where $M_{\lambda}^{(c)}$ is the generalized $c$-eigenspace of $z(0)$ on $M_{\lambda}$.
From $[z(0),x(n)]=[z,x](n)=jx(n)$ we get
$\big(z(0)-(c+j)\big)x(n)v=x(n)\big(z(0)-c\big)v$.
Then $\big(z(0)-(c+j)\big)^r x(n)v=x(n)\big(z(0)-c\big)^r v$ for all $r\ge1$,
so $x(n)M_{\lambda}^{(c)}\subseteq M_{\lambda-n}^{(c+j)}$ and
$z_n(0)x(n)v=x(n)z_n(0)v$. Hence $z_n(0)\in\End_{L_{\widehat{\mathfrak{g}}}(k,0)}(M)$.
\end{proof}

Let $\Oinf_{k}(\fgt)$ be the category of $\fgt$-modules of level $k$
with the Cartan subalgebra
$\fht$ acting semisimply and the positive part of $\fgt$ acting locally finitely.

\begin{thm}[\cite{GK25}]\label{thm:gk}
Let $\fg$ be a basic classical Lie superalgebra except $D(2|1;a)$ and
let $k$ be a boundary admissible level of $\fgt$. Then every weak $L_{\fgt}(k,0)$-module in
$\Oinf_{k}(\fgt)$ is completely reducible.
\end{thm}

\section{Finite semisimple category $\mathcal{O}_{k}^{ord}(\fg)$}

In this section,
let $k$ be a boundary admissible level,
we first prove the semisimplicity of the action of a Cartan subalgebra of $\fgt$ on ordinary $L_{\widehat{\mathfrak{g}}}(k,0)$-modules,
then we prove that the category $\mathcal{O}_{k}^{ord}(\fg)$ of ordinary $L_{\widehat{\mathfrak{g}}}(k,0)$-modules is finite and semisimple.

Let $\fg$ be a basic classical Lie superalgebra of type I and let $\alpha_{odd}$ be the odd isotropic simple root of $\fg$ (for
$\mathfrak{sl}(n|m)$ it is the simple root $\alpha_n$
and for $\mathfrak{osp}(2|2n)$ it is the simple root $\alpha_1$). Let $z$ be the
central element of $\fg_0$ which is normalized by $[z,x_{\alpha_{odd}}]=x_{\alpha_{odd}}$, and
let $\zeta\in\fh^{*}$ be the weight dual to $z$, i.e.,
$\zeta(h)=(z,h)$ for all $h\in\fh$, then
$(\zeta,\alpha_{odd})=\alpha_{odd}(z)=1$.

\begin{prop}\label{prop:gammap}
Let $\fg$ be a basic classical Lie superalgebra of type I and let $k=\tfrac{h^{\vee}}{u}-h^{\vee}$ be a boundary admissible level of $\fgt$.
For $p=0,\dots,u-1$, set $\gamma_p:=\alpha_{odd}+p\delta$, where $\delta$ is the imaginary root of affine Lie superalgebra $\fgt$. Then\\
{\rm(i)} $\gamma_p\in\Delta_{\widehat\lambda_p+\rh}^{+}$ is isotropic and indecomposable in
$\Delta_{\widehat\lambda_p+\rh}^{+}$;\\
{\rm(ii)} $(\widehat\lambda_p+\rh,\gamma_p)=0$;\\
{\rm(iii)} $(\zeta,\gamma_p)=1$.
\end{prop}

\begin{proof}
From the proof of Theorem 4.4 and Theorem 4.5 in \cite{LW26}, the admissible weight
$\widehat\lambda_p=t_{\beta}(k\Lambda_0+\rh)-\rh$ for some $\beta\in \mathfrak{h}^*$, where $t_{\beta}$ is the translation element of the affine Weyl group associated with $\beta$.
Then $\widehat\lambda_p=k\Lambda_0+(k+h^{\vee})\beta-((k\Lambda_0+\rh,\beta)+\frac{1}{2}(k+h^{\vee})(\beta,\beta))\delta.$
By Theorem \ref{thm:LW},
for $\fg=\mathfrak{sl}(n|m)$,
$\widehat\lambda_p=\Big(\frac{p+1}{u}-1\Big)h^{\vee}\,\Lambda_0-\frac{p\,h^{\vee}}{u}\Lambda_n$,
and for $\fg=\mathfrak{osp}(2|2n)$ one replaces $\Lambda_n$ by $\Lambda_1$.
Comparing the restrictions to $\fh$ of the two expressions of
$\widehat\lambda_p$ above, we get
$(k+h^{\vee})\beta=-\frac{p\,h^{\vee}}{u}\,\zeta,\text{i.e.}, \beta=-p\,\zeta$.
Since $t_{\beta}(\alpha_{odd})=\alpha_{odd}-(\beta,\alpha_{odd})\delta$ and
$(\zeta,\alpha_{odd})=1$, we obtain
$\gamma_p=t_{\beta}(\alpha_{odd})$.
Since $\Delta_{\widehat\lambda_p+\rh}=\Delta_{t_{\beta}(k\Lambda_0+\rh)}=t_{\beta}\Delta_{k\Lambda_0+\rh}$ and
$\alpha_{odd}$ belongs to the base $\widehat\Pi_u=\{u\delta-\theta,\alpha_1,\cdots,\alpha_n\}$ of $\Delta_{k\Lambda_0+\rh}^{+}$.
Then $\gamma_p=t_{\beta}(\alpha_{odd})$ belongs to the base $t_{\beta}\widehat\Pi_u$ of $\Delta_{\widehat\lambda_p+\rh}^+$,
hence $\gamma_p$ is indecomposable in
$\Delta_{\widehat\lambda_p+\rh}^{+}$.
Since $(\gamma_p,\gamma_p)=(\alpha_{odd},\alpha_{odd})+2p(\alpha_{odd},\delta)+p^2(\delta,\delta)=0$,
we have $\gamma_p$ is isotropic.
Since $t_{\beta}$ preserves the bilinear form, we have
$(\widehat\lambda_p+\rh,\gamma_p)
=\big(t_{\beta}(k\Lambda_0+\rh),\,t_{\beta}(\alpha_{odd})\big)
=(k\Lambda_0+\rh,\alpha_{odd})
=(\rho,\alpha_{odd})=0$ and
$(\zeta,\gamma_p)=(\zeta,\alpha_{odd})+p(\zeta,\delta)=(\zeta,\alpha_{odd})=1$.
\end{proof}

\begin{lem}\label{prop:hiorder}
Let $\fg$ be a basic classical Lie superalgebra of type I and let $k$ be a boundary admissible level of $\fgt$.
Let $M$ be an ordinary $L_{\widehat{\mathfrak{g}}}(k,0)$-module and let $z(0)=z_s(0)+z_n(0)$ be the Jordan-Chevalley
decomposition of $z(0)$ on $M$. If $z_n(0)\neq0$, then there exists an ordinary module $L(\widehat\lambda_p)$ and its
self-extension
$0\longrightarrow L(\widehat\lambda_p)\longrightarrow E_M\longrightarrow L(\widehat\lambda_p)\longrightarrow0$ in the category $\mathcal{O}_{k}^{ord}(\fg)$
such that the action of $z_n(0)$ on $E_M$ is non-zero, where the action of $z_n(0)$ on $E_M$ is induced from the action of $z_n(0)$ on $M$.
\end{lem}

\begin{proof}
From Lemma \ref{lem:Nend}, $z_n(0)\in\End_{L_{\widehat{\mathfrak{g}}}(k,0)}(M)$.
Since $z_n(0)$ preserves each finite-dimensional conformal weight space $M_{\lambda}$, it is
locally nilpotent.
From Proposition \ref{prop:finlen}, $M$ is of finite length.
Suppose the length of $M$ is $l$,
by the Fitting lemma \cite{AK74}, we have $M=\Ker\big(z_n(0)\big)^l\oplus\im\big(z_n(0)\big)^l$.
Since a locally nilpotent operator has no non-zero invertible part, we have $\im\big(z_n(0)\big)^l=0$, thus $z_n(0)$
is nilpotent on $M$. Let $r\ge2$ be an integer such that $z_n(0)^r=0$ and $z_n(0)^{r-1}\neq0$, and set
$M'=z_n(0)^{r-2}M$. As $z_n(0)\in\End_{L_{\widehat{\mathfrak{g}}}(k,0)}(M)$, $M'$ is a $L_{\widehat{\mathfrak{g}}}(k,0)$-submodule, and we have
$z_n(0)^2=0$ on $M'$ and $z_n(0)|_{M'}\neq0$.
Set $N=\Ker(z_n(0)|_{M'})$ and $W=M'/N$. Since $z_n(0)^2=0$ on $M'$, we have $z_n(0)M'\subseteq N$,
then for $v\in M'$, $z_n(0)v=0$ in $W$, so
$z_n(0)|_W=z_n(0)|_N=0$. Then
$z(0)$ acts semisimply on $W$, i.e., $z(0)=z_s(0)$ on $W$,
and $z(0)$ acts semisimply on $N$, i.e.,
$z(0)=z_s(0)$ on $N$, by
Lemma~\ref{lem:central}, $\fh^{\mathrm{ss}}(0)$ acts semisimply on $W$ and
$N$. Note that the center $K$ acts as the scalar $k$ and the derivation $d=-L(0)$ acts semisimply on $W$ and $N$.
Hence $\fht(0)$ acts semisimply on $W$ and $N$. Moreover, $W$
and $N$ are subquotients of the ordinary module $M$, hence the positive part acts locally finitely on $W$ and $N$. Thus $W,N\in\Oinf_k(\fgt)$, and by
Theorem~\ref{thm:gk} both $W$ and $N$ are semisimple $L_{\widehat{\mathfrak{g}}}(k,0)$-modules.
Define $\phi:W\to z_n(0)M'$ by $\phi(v+N)=z_n(0)v$.
Note that $N=\Ker z_n(0)|_{M'}$,
this is well defined.
It is clear that $\phi$ is an isomorphism and then $W\cong z_n(0)M'\subseteq N$.
Since $z_n(0)|_{M'}\neq0$, we have $W\neq0$. As $W$ is a semisimple $L_{\widehat{\mathfrak{g}}}(k,0)$-module, we choose a
simple direct summand $J\subseteq W$, by Theorem~\ref{thm:LW}, $J\cong L(\widehat\lambda_p)$ for some
$p$. Let $T=\phi(J)\subseteq N$, then $T$ is simple and $T\cong L(\widehat\lambda_p)$. Since $N$ is
semisimple, we choose a $L_{\widehat{\mathfrak{g}}}(k,0)$-submodule $N'\subseteq N$ such that $N=T\oplus N'$. Let
$\pi:M'\to W=M'/N$ be the canonical map and set $M'_J=\pi^{-1}(J)$, then
$0\to N\to M'_J\to J\to0$ is exact. Note that $N/N'\cong T$, we obtain a
short exact sequence
$0\longrightarrow T\longrightarrow M'_J/N'\longrightarrow J\longrightarrow0$,
where $T\cong J\cong L(\widehat\lambda_p)$ and
$M'_J/N'$ is an ordinary module, we denote $E_M=M'_J/N'$.
Since
$z_n(0)(N)=0\subseteq N'$, $z_n(0)$ preserves $M'_J$ and descends to $E_M$, denote
the descended operator by ${z_n^{E_M}(0)}$. For $t=\phi(s)\in T$ with $s\in J$, we choose
$v\in M'_J$ such that $\pi(v)=s$, then ${z_n^{E_M}(0)}(v+N')=z_n(0)v=t$, so
${z_n^{E_M}(0)}(E_M)\supseteq T\neq0$ and then ${z_n^{E_M}(0)}\neq0$. Since
$z_n(0)^2=0$ on $M'$, we have ${z_n^{E_M}(0)}^2=0$ on $E_M$. As $z(0)|_{E_M}=z_s(0)|_{E_M}+{z_n^{E_M}(0)}$ with $z_s(0)|_{E_M}$ semisimple,
${z_n^{E_M}(0)}$ is the nilpotent part of $z(0)$ on $E_M$, and it is non-zero.
\end{proof}


\begin{prop}\label{thm:typeI}
Let $\fg$ be a basic classical Lie superalgebra of type I and let $k$ be a
boundary admissible level of $\fgt$. Then $\fh(0)$
acts semisimply on ordinary $L_{\fgt}(k,0)$-modules.
\end{prop}

\begin{proof}
Let $M$ be an ordinary $L_{\widehat{\mathfrak{g}}}(k,0)$-module. By Lemma~\ref{lem:central} it suffices to show
that $z(0)$ is semisimple on $M$. Suppose not, the nilpotent part $z_n(0)$ of
$z(0)=z_s(0)+z_n(0)$ is non-zero. By Proposition~\ref{prop:finlen}, $M$ is of finite length,
and by Lemma~\ref{lem:Nend}, $z_n(0)\in\End_{L_{\widehat{\mathfrak{g}}}(k,0)}(M)$. From Lemma~\ref{prop:hiorder}, there
exists a simple ordinary module $L(\widehat\lambda_p)$ and a self-extension
$0\longrightarrow L(\widehat\lambda_p)\longrightarrow E_M\longrightarrow L(\widehat\lambda_p)\longrightarrow0$ in the category $\mathcal{O}_{k}^{ord}(\fg)$
on which $z_n(0)$ on $E_M$ is non-zero.
Let $v\in L(\widehat\lambda_p)$ be a highest weight vector and let $\widetilde v\in E_M$ be a lift of the
highest weight vector of the quotient $E_M/L(\widehat\lambda_p)$ such that
$h(0)\,\widetilde v=\widehat\lambda_p(h)\,\widetilde v+\nu(h)\,v$ for any $h\in\fht$,
i.e., $\Upsilon_{\widehat{\lambda}_p}(E_M)=\nu$.
Since $\fh^{\mathrm{ss}}(0)$ acts semisimply on $E_M$ by
Lemma~\ref{lem:central}, we have
$\nu|_{\fh^{\mathrm{ss}}}=0$. Similarly we also have
$\nu(K)=0$ and $\nu(d)=0$. Note that the annihilator of
$\fh^{\mathrm{ss}}\oplus\C K\oplus\C d$ in $\fht^*$ is one-dimensional, spanned by
$\zeta$ (Recall that $\zeta$ is dual to the center $z$). Hence
$\nu=c\,\zeta$ for some $c\in\C$.
Since the nilpotent part of $z(0)$ on $E_M$ being non-zero is equivalent
to $z(0)$ having a non-trivial Jordan block on the two-dimensional space
$\C v\oplus\C\widetilde v$, i.e., $\nu(z)\neq0$, and
$\zeta(z)\neq0$, we obtain $c\neq0$.
On the other hand,
by Proposition~\ref{prop:gammap}, $\gamma_p=\alpha_{odd}+p\delta$ is isotropic and
indecomposable in $\Delta_{\widehat\lambda_p+\rh}^{+}$ such that
$(\widehat\lambda_p+\rh,\gamma_p)=0$, applying Lemma~\ref{lem:tangent} to the
self-extension $E_M$, we get
$(\nu,\gamma_p)=0$. By Proposition~\ref{prop:gammap},
$(\nu,\gamma_p)=c\,(\zeta,\gamma_p)=c,$
so $c=0$, a contradiction. Therefore $z_n(0)=0$ and $z(0)=z_s(0)$ is semisimple on $M$. By
Lemma~\ref{lem:central}, $\fh(0)=\fh^{\mathrm{ss}}(0)\oplus\C z(0)$ acts semisimply on $M$.
\end{proof}

Next we consider the case of finite-dimensional simple Lie algebras and the basic classical Lie superalgebras of type II.

\begin{prop}\label{prop:typeII}
Let $\fg$ be a finite-dimensional simple Lie algebra or a basic classical Lie
superalgebra of type II except $D(2|1;a)$ and let $k$ be a boundary admissible level of $\fgt$. Then $\fh(0)$ acts semisimply on ordinary $L_{\fgt}(k,0)$-modules.
\end{prop}

\begin{proof}
Let $\fg$ be a finite-dimensional simple Lie algebra or a basic classical Lie
superalgebra of type II, then the even part $\fg_{\bar{0}}$ is semisimple. For each
ordinary $L_{\fgt}(k,0)$-module $M$, the conformal weight space $M_{\lambda}$ is a
finite-dimensional $\fg$-module, hence $M_{\lambda}$ is a completely reducible
$\fg_{\bar{0}}$-module by Weyl's completely reducibility theorem, thus $\fh(0)$ acts
semisimply on each $M_{\lambda}$. Hence $\fh(0)$ acts semisimply on ordinary $L_{\fgt}(k,0)$-module $M$.
\end{proof}

Recall from \cite{EGNO15} that a $\mathbb{C}$-linear abelian category $\mathcal{C}$ is finite if the following conditions hold:
(i) $\mathcal{C}$ has finite dimensional spaces of morphisms; (ii) every object of $\mathcal{C}$ has finite length;
(iii) $\mathcal{C}$ has enough projectives, i.e., every simple object of $\mathcal{C}$ has a projective cover;
(iv) there are finitely many isomorphism classes of simple objects.
Combining Proposition \ref{thm:typeI} and Proposition \ref{prop:typeII}, we obtain the following theorem.

\begin{thm}\label{cor:ordsemisimple}
Let $\fg$ be a basic classical Lie superalgebra except $D(2|1;a)$ and
let $L_{\widehat{\mathfrak{g}}}(k,0)$ be the simple affine vertex operator superalgebra with
boundary admissible level $k$.
Then the category $\mathcal{O}_{k}^{ord}(\fg)$ of ordinary $L_{\widehat{\mathfrak{g}}}(k,0)$-modules is finite and semisimple.
\end{thm}

\begin{proof}
Let $M$ be an ordinary $L_{\widehat{\mathfrak{g}}}(k,0)$-module.
By Proposition \ref{thm:typeI} and Proposition \ref{prop:typeII}, $\fh(0)$ acts semisimply on $M$.
Since $K$ acts as the scalar $k$ and $d=-L(0)$ acts semisimply, $\fht(0)$ acts semisimply on $M$. As in the
proof of Lemma~\ref{prop:hiorder},  the positive
part of $\fgt$ acts locally finitely on $M$, hence $M\in\Oinf_k(\fgt)$, and then $M$ is
completely reducible by Theorem~\ref{thm:gk}. Thus $\mathcal{O}_{k}^{ord}(\fg)$ is semisimple and has enough projectives.
By Theorem~\ref{thm:LW}, there are finitely many isomorphism classes of simple objects in $\mathcal{O}_{k}^{ord}(\fg)$.
By Proposition~\ref{prop:finlen}, $M$ has
finite length, so it is a finite direct sum of irreducible modules.
Then $\mathcal{O}_{k}^{ord}(\fg)$ has finite dimensional spaces of morphisms, thus $\mathcal{O}_{k}^{ord}(\fg)$ is finite.
\end{proof}

\section{Rigid braided tensor supercategory $\mathcal{O}_{k}^{ord}(\fg)$}

In this section, we prove that the category $\mathcal{O}_{k}^{ord}(\fg)$ of ordinary $L_{\widehat{\mathfrak{g}}}(k,0)$-modules is a
rigid braided tensor supercategory with ribbon twist $\theta=e^{2\pi iL(0)}$, and thus it is a ribbon fusion supercategory.

\subsection{Braided tensor supercategory}

Let $\KL_k(\fg)$ denote the category of finite-length generalized
$L_{\widehat{\mathfrak{g}}}(k,0)$-modules of level $k$. A generalized
module belongs to $\KL_k(\fg)$ if and only if it admits a finite composition
series whose composition factors are irreducible ordinary modules \cite{CY}. In particular,
the irreducible objects of $\KL_k(\fg)$ are exactly the irreducible ordinary
modules.

\begin{thm}[\cite{CY}]\label{thm:CY}
Assume that the category of ordinary $L_{\fgt}(k,0)$-modules is semisimple. Then
$\KL_k(\fg)$ has a natural structure of braided tensor category.
\end{thm}

Theorem~\ref{thm:CY} was proved in \cite{CY} for affine vertex operator algebras. As
mentioned in \cite[{\S}1.5]{CY} and \cite[{\S}2.3]{CMY22}, the results naturally generalize to affine vertex operator superalgebras.
We next show that self-extensions of irreducible ordinary modules vanish in
$\KL_k(\fg)$.

\begin{lem}\label{lem:extKL}
Let $\fg$ be a basic classical Lie superalgebra except $D(2|1;a)$ and
let $k$ be a boundary admissible level of $\fgt$.
Then $\Ext^1_{\KL_k(\fg)}(L(\widehat\lambda_p),L(\widehat\lambda_p))=0$
for any irreducible ordinary $L_{\widehat{\mathfrak{g}}}(k,0)$-module $L(\widehat\lambda_p)$.
\end{lem}

\begin{proof}
Let
$0\longrightarrow L(\widehat\lambda_p)\longrightarrow E\longrightarrow L(\widehat\lambda_p)\longrightarrow0$
be an exact sequence in $\KL_k(\fg)$.
Since $E\in\KL_k(\fg)$, $E$ is grading-restricted. For any $h\in\fh$, $[L(0),h(0)]=0$, so $h(0)$ preserves each
generalized $L(0)$-eigenspace, hence $\fh(0)$ acts locally finitely on $E$.
Let $h_p$ be the lowest conformal weight of $L(\widehat\lambda_p)$. We view $E$ as a
$\fgt$-module by setting $d=-L(0)+h_p$ on $E$, then $E$ is a $\fgt$-module with generalized highest weight $\widehat\lambda_p$.
Thus there exists a unique $\nu\in\fht^*$ with
$\Upsilon_{\widehat\lambda_p}(E)=\nu$.
Since $K$
acts on $E$ by $k$, we have $\nu(K)=0$.
Let $E_{h_p}$ be the
generalized $L(0)$-eigenspace of eigenvalue $h_p$, then $E_{h_p}$ is a finite-dimensional $\fg_{\bar{0}}$-module.

If $\fg$ is a basic classical Lie
superalgebra of type II or a finite-dimensional simple Lie algebra, then $\fg_{\bar{0}}=\fg_{\bar{0}}^{\mathrm{ss}}$ is
semisimple, so by Weyl's completely reducibility theorem, $\fh(0)$ acts semisimply on $E_{h_p}$,
 hence $\nu|_{\fh}=0$, together with $\nu(K)=0$ we have $\nu\in\C\delta$.
If $\fg$ is a basic classical Lie
superalgebra of type I, then $\fg_{\bar{0}}=\fg_{\bar{0}}^{\mathrm{ss}}\oplus\C z$
and $\fh^{\mathrm{ss}}(0)$ acts semisimply on $E_{h_p}$, so $\nu|_{\fh^{\mathrm{ss}}}=0$,
together with $\nu(K)=0$, we have $\nu=a\delta+c\,\zeta$ for some
$a,c\in\C$. By Proposition~\ref{prop:gammap} and
Lemma~\ref{lem:tangent}, we have
$(\nu,\gamma_p)=0$. On the other hand,
 by Proposition~\ref{prop:gammap}(iii), $(\delta,\gamma_p)=0$ and $(\zeta,\gamma_p)=1$, so
$0=(\nu,\gamma_p)=(a\delta+c\zeta,\gamma_p)=c$,
and thus $\nu=a\delta\in\C\delta$.
Hence in both cases $\nu\in\C\delta$.
Note that $\Delta^{\perp}=\C\delta$ for indecomposable affine Lie superalgebra $\fgt$.
By Lemma~\ref{lem:split}, the exact sequence $0\longrightarrow L(\widehat\lambda_p)\longrightarrow E\longrightarrow L(\widehat\lambda_p)\longrightarrow0$
splits over $[\fgt,\fgt]$, thus $0\longrightarrow L(\widehat\lambda_p)\longrightarrow E\longrightarrow L(\widehat\lambda_p)\longrightarrow0$
splits over $L_{\widehat{\mathfrak{g}}}(k,0)$. Hence
$\Ext^1_{\KL_k(\fg)}(L(\widehat\lambda_p),L(\widehat\lambda_p))=0$.
\end{proof}

\begin{thm}\label{thm:KL}
Let $\fg$ be a basic classical Lie superalgebra except $D(2|1;a)$ and
let $L_{\widehat{\mathfrak{g}}}(k,0)$ be the simple affine vertex operator superalgebra with
boundary admissible level $k$.
Then $\mathcal{O}_{k}^{ord}(\fg)=\KL_k(\fg)$ and $\mathcal{O}_{k}^{ord}(\fg)$ is a braided tensor supercategory.
\end{thm}

\begin{proof}
Let $M\in\KL_k(\fg)$ and $L(0)=L(0)_s+L(0)_n$ be the Jordan-Chevalley
decomposition of $L(0)$ on $M$. For a
homogeneous element $v\in L_{\widehat{\mathfrak{g}}}(k,0)$ with conformal weight $\mathrm{wt}(v)$, $[L(0)_s,v_m]=(\mathrm{wt}(v)-m-1)v_m=[L(0),v_m]$,
hence $[L(0)_n,v_m]=0$ for all homogeneous $v$ and all $m$.
Then $L(0)_n\in\End_{L_{\widehat{\mathfrak{g}}}(k,0)}(M)$.
Since $M$
has finite length, similar to Lemma \ref{prop:hiorder}, $L(0)_n$ is nilpotent. Suppose $L(0)_n\neq0$, let $r\ge2$ be an integer such that $L(0)_n^r=0$ and
$L(0)_n^{r-1}\ne0$, and set
$M'=L(0)_n^{r-2}M$. Then $L(0)_n^2=0$ on $M'$ and $L(0)_n|_{M'}\neq0$. Set
$N=\Ker(L(0)_n|_{M'})$ and $W=M'/N$. Since $L(0)_n^2=0$ on $M'$, we have $L(0)_nM'\subseteq N$ and
$L(0)_n$ acts as zero on both $W$ and $N$, thus $L(0)$ acts semisimply on $W$ and $N$,
so $W$ and $N$ are ordinary $L_{\widehat{\mathfrak{g}}}(k,0)$-modules. By Theorem~\ref{cor:ordsemisimple}, $W$ and
$N$ are semisimple $L_{\widehat{\mathfrak{g}}}(k,0)$-modules.
As in the proof of Lemma~\ref{prop:hiorder}, the map
$\phi:W\to L(0)_nM'\subseteq N$, $\phi(y+N)=L(0)_ny$, is an isomorphism. Choose a simple
direct summand $J\subseteq W$ and set $T=\phi(J)\subseteq N$, then $T\cong J$ is
simple. Let $N=T\oplus N'$ and $E=\pi^{-1}(J)/N'$, where $\pi:M'\to W=M'/N$ is the canonical
map, we obtain a short exact sequence $0\longrightarrow T\longrightarrow E\longrightarrow J\longrightarrow0$
in $\KL_k(\fg)$ on which $L(0)_n$ acts on $E$ non-zero. By Theorem~\ref{thm:LW}, $T\cong J\cong L(\widehat\lambda_p)$ for some
$p$, so this is a non-trivial self-extension of the irreducible ordinary module
$L(\widehat\lambda_p)$, this contradicts with Lemma \ref{lem:extKL}. Thus $L(0)_n=0$, i.e., $L(0)$ is semisimple on
$M$, thus $M$ is ordinary.

Conversely, by Theorem~\ref{cor:ordsemisimple} every ordinary module is a finite
direct sum of irreducible ordinary modules, hence every ordinary module lies in $\KL_k(\fg)$. Thus
$\KL_k(\fg)=\mathcal{O}_{k}^{ord}(\fg)$, and by Theorem~\ref{thm:CY}, $\mathcal{O}_{k}^{ord}(\fg)$ is a braided tensor supercategory.
\end{proof}

\begin{rem}
{\em Note that for any level $k\ne-h^\vee$, every irreducible ordinary $L_{\widehat{\mathfrak{g}}}(k,0)$-module is $C_1$-cofinite \cite{CHY18}.
For any ordinary $L_{\widehat{\mathfrak{g}}}(k,0)$-module $M$, it is a finite direct sum of irreducible modules,
then $M$ is $C_1$-cofinite.
By \cite[Theorem~1.2]{Huang25} (see also \cite[Theorem~2.25]{CMOY}), the category $\mathcal{D}$ of $C_1$-cofinite grading-restricted generalized $L_{\widehat{\mathfrak{g}}}(k,0)$-modules
is a braided tensor supercategory with twist
$\theta=e^{2\pi iL(0)}$.
For any two irreducible ordinary modules $M_1,M_2$, denote by $M_1\boxtimes_{P(z)} M_2$ the $P(z)$-tensor product in $\mathcal{D}$.
Since $\mathcal{O}_{k}^{ord}(\fg)$ is finite and semisimple, from Corollary 3.8 of \cite{CY}, we have $M_1\boxtimes_{P(z)} M_2\in\KL_k(\fg)=\mathcal{O}_{k}^{ord}(\fg)$,
hence $\mathcal{O}_{k}^{ord}(\fg)$ is closed under the $P(z)$-tensor products.
Thus we also obtain that $\mathcal{O}_{k}^{ord}(\fg)$ is a braided tensor supercategory with twist
$\theta=e^{2\pi iL(0)}$.}
\end{rem}

Let $\mathcal{O}_{k}^{ord}(\fg)^{\mathrm{ev}}$ be the subcategory of $\mathcal{O}_{k}^{ord}(\fg)$ with the same objects and with parity-preserving homomorphisms as morphisms.
Since all structure morphisms of $\mathcal{O}_{k}^{ord}(\fg)$ are even \cite{CKM},
the even part $\mathcal{O}_{k}^{ord}(\fg)^{\mathrm{ev}}$
is a braided tensor category. By Theorem~\ref{cor:ordsemisimple}, $\mathcal{O}_{k}^{ord}(\fg)^{\mathrm{ev}}$ is also finite and semisimple.

\subsection{Rigidity}

Recall from \cite{BD} that an $r$-category is a monoidal category $\mathcal{C}$ such
that for every object $X\in\mathcal{C}$, the functor $Z\mapsto\Hom(Z\boxtimes X,\mathbf{1})$ is
representable by some object $X^{*}$, and the functor
$^{*}:\mathcal{C}\to\mathcal{C}^{\mathrm{mop}}$ (monoidal and arrow opposite) is an equivalence (cf. \cite{EP}).

\begin{prop}\label{prop:rcat}
Let $\fg$ be a basic classical Lie superalgebra except $D(2|1;a)$ and
let $k$ be a boundary admissible level of $\fgt$.
Then the braided tensor category $\mathcal{O}_{k}^{ord}(\fg)^{\mathrm{ev}}$ is an $r$-category.
\end{prop}

\begin{proof}
Set $\mathcal{C}=\mathcal{O}_{k}^{ord}(\fg)^{\mathrm{ev}}$.
For any ordinary $L_{\widehat{\mathfrak{g}}}(k,0)$-module $M=\bigoplus_{\lambda\in\mathbb{C}}M_{\lambda}$, let
$M'=\bigoplus_{\lambda\in\mathbb{C}}M_{\lambda}^{*}$ be the graded dual, then each $M_{\lambda}^{*}$ is
finite-dimensional and $L(0)$ acts semisimply on $M'$, hence $M'\in\mathcal{C}$.
From Proposition 5.3.1 of \cite{FHL93}, we have $M\cong M''$, so the contravariant functor
$D(M)=M'$ satisfies $D^2\cong\mathrm{id}_{\mathcal{C}}$.
Note that the unit object $L_{\widehat{\mathfrak{g}}}(k,0)$ is self-dual.
Let $\Ic\!\left(\begin{smallmatrix}M_3\\ M_1\ M_2\end{smallmatrix}\right)$
denote the space of parity-preserving intertwining operators of type
$\left(\begin{smallmatrix}M_3\\ M_1\ M_2\end{smallmatrix}\right)$,
where $M_1,M_2,M_3\in\mathcal{C}$. By the universal property of the vertex tensor product
\cite{HL95-1}, there is a natural isomorphism between
$\Hom_{\mathcal{C}}^0(M_1\boxtimes M_2,M_3)$
and $\Ic\!\left(\begin{smallmatrix}M_3\\ M_1\ M_2\end{smallmatrix}\right)$,
where $\Hom_{\mathcal{C}}^{0}(M_1\boxtimes M_2,M_3)$ is the space of parity-preserving homomorphisms from $M_1\boxtimes M_2$ to $M_3$.
Note that
$\Ic\!\left(\begin{smallmatrix}M_3\\ M_1\ M_2\end{smallmatrix}\right)
\cong\Ic\!\left(\begin{smallmatrix}M_2'\\ M_1\ M_3'\end{smallmatrix}\right)$.
Fix $M,N\in\mathcal{C}$, since $L_{\widehat{\mathfrak{g}}}(k,0)'\cong L_{\widehat{\mathfrak{g}}}(k,0)$ and
$N\boxtimes L_{\widehat{\mathfrak{g}}}(k,0)\cong N$, we have
\begin{align*}
 \Hom_{\mathcal{C}}^0(N\boxtimes M,L_{\widehat{\mathfrak{g}}}(k,0))
&\cong\Ic\!\left(\begin{smallmatrix}L_{\widehat{\mathfrak{g}}}(k,0)\\ N\ M\end{smallmatrix}\right)
\cong\Ic\!\left(\begin{smallmatrix}M'\\ N\ L_{\widehat{\mathfrak{g}}}(k,0)'\end{smallmatrix}\right)
\cong\Ic\!\left(\begin{smallmatrix}M'\\ N\ L_{\widehat{\mathfrak{g}}}(k,0)\end{smallmatrix}\right)\\
&\cong\Hom_{\mathcal{C}}^0(N\boxtimes L_{\widehat{\mathfrak{g}}}(k,0),M')
\cong\Hom_{\mathcal{C}}^0(N,M').
\end{align*}
Hence the functor
$N\mapsto\Hom_{\mathcal{C}}^0(N\boxtimes M,L_{\widehat{\mathfrak{g}}}(k,0))$ is represented by $M'$. Since
$D^2\cong\mathrm{id}_{\mathcal{C}}$, the contravariant functor $D$ is an equivalence, thus $\mathcal{C}$ is an $r$-category.
\end{proof}

\begin{thm}[\cite{EP}]\label{thm:EP}
Every semisimple braided $r$-category of moderate growth is rigid.
\end{thm}

\begin{thm}\label{thm:rigid}
Let $\fg$ be a basic classical Lie superalgebra except $D(2|1;a)$ and let
$L_{\widehat{\mathfrak{g}}}(k,0)$ be the simple affine vertex operator superalgebra with
boundary admissible level $k$. Then the
category $\mathcal{O}_{k}^{ord}(\fg)$ of ordinary $L_{\widehat{\mathfrak{g}}}(k,0)$-modules is a finite semisimple rigid braided tensor supercategory with twist
$\theta=e^{2\pi iL(0)}$.
\end{thm}
\begin{proof}
Note that $\mathcal{O}_{k}^{ord}(\fg)^{\mathrm{ev}}$ is a finite semisimple braided tensor category with twist
$\theta=e^{2\pi iL(0)}$.
By Proposition~\ref{prop:rcat},
$\mathcal{O}_{k}^{ord}(\fg)^{\mathrm{ev}}$ is an $r$-category.
Since the category $\mathcal{O}_{k}^{ord}(\fg)^{\mathrm{ev}}$ is finite and semisimple, from \cite{EP},
it is a category of moderate growth. By Theorem~\ref{thm:EP}, $\mathcal{O}_{k}^{ord}(\fg)^{\mathrm{ev}}$ is rigid,
thus $\mathcal{O}_{k}^{ord}(\fg)$ is rigid.
\end{proof}

Recall that a fusion category is a finite semisimple rigid tensor category
and a ribbon category is a rigid braided tensor category equipped with a ribbon twist \cite{EGNO15}.
From Theorem 4.1 of \cite{H08}, the twist
$\theta=e^{2\pi iL(0)}$ in the
category $\mathcal{O}_{k}^{ord}(\fg)$ is a ribbon twist. Thus we have the following corollary.

\begin{cor}
Let $\fg$ be a basic classical Lie superalgebra except $D(2|1;a)$ and let
$L_{\widehat{\mathfrak{g}}}(k,0)$ be the simple affine vertex operator superalgebra with
boundary admissible level $k$. Then the
category $\mathcal{O}_{k}^{ord}(\fg)$ is a ribbon fusion supercategory.
\end{cor}

\section*{Acknowledgements}
We would like to thank Maria Gorelik for valuable discussions on the semisimplicity of a Cartan subalgebra action on ordinary $L_{\widehat{\mathfrak{g}}}(k,0)$-modules.
We also thank Jinwei Yang for drawing our attention to the paper \cite{EP}.
Qing Wang is supported by China NSF grant No. 12571033.

\end{document}